\documentclass[oneside,leqno,11pt]{amsart}

\title{Ioana's superrigidity theorem and orbit equivalence relations}

\author{Samuel Coskey}

\address{Department of Mathematics\\Boise State University\\1910 University Dr\\Boise, ID\ \ 83725}
\email{scoskey@nylogic.org}
\urladdr{boolesrings.org/scoskey}

\usepackage{mydefs}
\usepackage[marginratio=1:1]{geometry}
\usepackage{setspace}
\usepackage{mathpazo,bm}

\usepackage{diagrams}
\newarrow{Mapsto}|--->
\newarrow{DMapsto}|{dash}{}{dash}>
\theoremstyle{remark}
\newtheorem*{claimI}{Claim I}
\newtheorem*{claimII}{Claim II}
\DeclareMathOperator{\QEnd}{QEnd}

\DeclareMathOperator{\End}{End}
\DeclareMathOperator{\stab}{stab}
\newcommand{\PSL}{\mathop{\mathrm{PSL}}\nolimits}
\newcommand{\qiso}{\sim}
\newcommand{\oqiso}{\ord\qiso}

\begin{document}
\maketitle

\begin{abstract}
  In this expository article, we give a survey of Adrian Ioana's
  cocycle superrigidity theorem for profinite actions of Property~(T)
  groups, and its applications to ergodic theory and set theory.  In
  addition to a statement and proof of Ioana's theorem, this article
  features:
  \begin{itemize}
  \item An introduction to rigidity, including a crash course in Borel
    cocycles and a summary of some of the best-known superrigidity
    theorems;
  \item Some easy applications of superrigidity, both to ergodic
    theory (orbit equivalence) and set theory (Borel reducibility);
    and
  \item A streamlined proof of Simon Thomas's theorem that the
    classification of torsion-free abelian groups of finite rank is
    intractable.
  \end{itemize}  
\end{abstract}

\section{Introduction}

In the past fifteen years superrigidity theory has had a boom in the
number and variety of new applications.  Moreover, this has been
coupled with a significant advancement in techniques and results.  In
this article, we survey one such new result, namely Ioana's theorem on
profinite actions of property (T) groups, and some of its applications
in ergodic theory and in set theory.  In the concluding section we
highlight an application to the classification problem for
torsion-free abelian groups of finite rank.  The narrative is strictly
expository, with most of the material adapted from the work of Adrian
Ioana, myself, and Simon Thomas.

Although Ioana's theorem is relatively recent, it will be of interest
to readers who are new to rigidity because the proof is natural and
there are many immediate applications.  Therefore, we have taken care
to keep the non-expert in mind.  We do assume that the reader is
familiar with the notion of ergodicity of a measure-preserving action
and with unitary representations of countable groups.  We will not go
into great detail on Property~(T), since for our purposes it is enough
to know that $\SL_n(\ZZ)$ satisfies Property~(T) when $n>2$.  Rather,
we shall introduce it just when it's needed, and hopefully its key
appearance in the proof of Ioana's theorem will provide some insight
into its meaning.

The concept of superrigidity was introduced by Mostow and Margulis in
the context of studying the structure of lattices in Lie groups.
Here, $\Gamma$ is said to be a \emph{lattice} in the (real) Lie group
$G$ if it is discrete and $G/\Gamma$ admits an invariant probability
measure.  Very roughly speaking, Margulis showed that if $\Gamma$ is a
lattice in a simple (higher-rank) real Lie group $G$, then any
homomorphism from $\Gamma$ into an algebraic group $H$ lifts to an
algebraic map from $G$ to $H$.  This implies Mostow's theorem, which
states that any isomorphic lattices $\Gamma,\Lambda$ in a simple
(higher-rank) Lie group $G$ must be conjugate inside $G$.

We will leave this first form of rigidity on the back burner and
primarily consider instead a second form, initially considered by
Zimmer, which is concerned with group \emph{actions}.  (The connection
between the two forms of rigidity is that both can be cast in terms of
measurable cocycles, which will be introduced in the next section.
For the connection between cocycles and lifting homomorphisms, see
\cite[Example~4.2.12]{zimmer}.)  The basic notions are as follows.
Two probability measure-preserving actions $\Gamma\actson X$ and
$\Lambda\actson Y$ are said to be \emph{orbit equivalent} if there
exists a measure-preserving almost bijection $f\from X\into Y$ such
that $\Gamma x=\Gamma x'$ iff $\Lambda f(x)=\Lambda f(x')$.  They are
said to be \emph{isomorphic} if additionally there exists an
isomorphism $\phi\from\Gamma\into\Lambda$ such that $f(\gamma
x)=\phi(\gamma)f(x)$.  Essentially, Zimmer showed that any
(irreducible) ergodic action $\Gamma\actson X$ of a lattice in a
(higher rank) simple Lie group is \emph{superrigid} in the sense that
it cannot be orbit equivalent to another action of an algebraic group
$\Lambda\actson Y$ without being isomorphic to it.  (For elementary
reasons it is necessary to assume that $\Lambda$ acts freely on $Y$.)
See \cite[Theorem~5.2.1]{zimmer} for a weak statement of this result
and \cite[Section~1]{furman-oe} for further discussion.

It is natural to ask whether there exists an analog of Zimmer's
theorem in the context of general measure-preserving actions, that is,
with the algebraic hypothesis on $\Lambda$ removed.  Many rigidity
results have been established along these lines (for instance, see
\cite{furman-oe}, \cite{monodshalom}, \cite{kida-oe}).  One of the
landmark results in this direction was obtained recently by Popa
\cite{popa_paper}, who found a large class of measure-preserving
actions $\Gamma\actson X$ which are superrigid in the general sense
that $\Gamma\actson X$ cannot be orbit equivalent with another (free)
action without being isomorphic to it.  Specifically, his theorem
states that if $\Gamma$ is a Property~(T) group, then the free part of
its left-shift action on $X=2^\Gamma$ (the so-called Bernoulli action)
is an example of a superrigid action.  Following on Popa's work,
Ioana's theorem gives a second class of examples of superrigid
actions, namely the profinite actions of Property~(T) groups.

This article is organized as follows.  The second section gives some
background on Borel cocycles, a key tool in rigidity theory.  A
slightly weakened version of Ioana's theorem is stated in the third
section.  The proof itself is split between Section~4, which contains
a general purpose lemma, and Section~5, which contains the heart of
the argument.  Although these are largely unchanged from Ioana's own
account, I have inserted many additional remarks to smooth the
experience for the newcomer.

In Section~6 we give a couple of the easier applications of the main
theorem.  First, we show how to obtain many orbit inequivalent
profinite actions of $\SL_n(\ZZ)$.  We also explore applications to
logic and set theory by considering Borel reducibility.  In
particular, we point out some of the extra challenges one faces when
working in the purely set-theoretic (\emph{i.e.}, Borel) context, as
opposed to the more familiar measure context.

Finally, in the last section we use Ioana's theorem to give a
self-contained and slightly streamlined proof of Thomas's theorem
that the complexity of the isomorphism problem for torsion-free
abelian groups of finite rank increases strictly with the rank.

\section{Rigidity via cocycles}

We begin by introducing a slightly more expansive notion of orbit
equivalence rigidity.  If $\Gamma\actson X$ and $\Lambda\actson Y$ are
arbitrary Borel actions of countable groups, then a function $f\from
X\into Y$ is said to be a \emph{homomorphism of orbits} if $\Gamma
x=\Gamma x'$ implies $\Lambda f(x)=\Lambda f(x')$.  It is said to be a
\emph{homomorphism of actions} if additionally there exists a
homomorphism $\phi\from\Gamma\into\Lambda$ such that $f(\gamma
x)=\phi(\gamma)f(x)$.  (Note that these terms are not exactly
standard.)  Informally, we shall say that $\Gamma\actson X$ is
\emph{superrigid} if whenever $\Lambda\actson Y$ is a free action and
$f\from X\into Y$ is a homomorphism of orbits, then $f$ in fact arises
from a homomorphism of actions (that is, $f$ is equivalent to a
homomorphism of actions in a sense defined below).

Following Margulis and Zimmer, we shall require the language of Borel
cocycles to describe and prove superrigidity results.  A
\emph{cocycle} is an object which is associated with a given
homomorphism of orbits $f\from X\into Y$ as follows.  Observe that for
every $(\gamma,x)\in\Gamma\times X$, there exists a
$\lambda\in\Lambda$ such that $f(\gamma x)=\lambda f(x)$.  Moreover,
$\Lambda$ acts freely on $Y$ iff this $\lambda$ is always uniquely
determined by the data $f$, $\gamma$ and $x$.  In other words, in this
case $f$ determines a function $\alpha\from\Gamma\times X\into\Lambda$
which satisfies
\[f(\gamma x)=\alpha(\gamma,x)f(x)
\]
This map is called the \emph{cocycle corresponding to $f$}, and it is
easy to see that it is Borel whenever $f$ is.  Moreover, the cocycle
$\alpha$ satisfies the composition law
$\alpha(\gamma'\gamma,x)=\alpha(\gamma',\gamma x)\alpha(\gamma,x)$;
this is called the \emph{cocycle condition}.  See
Figure~\ref{fig_comp} for a visual depiction of the cocycle condition.

\begin{figure}[ht]
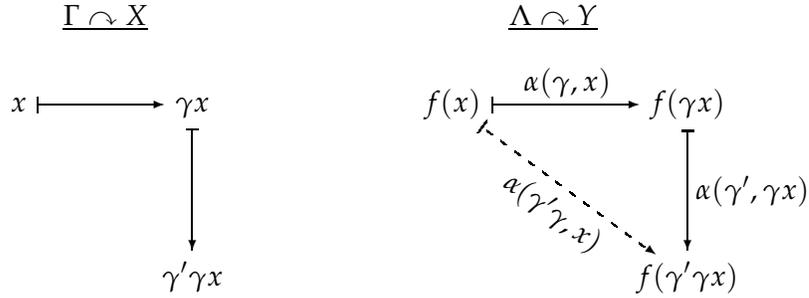

\begin{diagram}
  & \underline{\Gamma\actson X} & &&& & \underline{\Lambda\actson Y}\\
x & \rMapsto & \gamma x        &&& f(x) & \rMapsto^{\quad\alpha(\gamma,x)}
                                         & f(\gamma x) \\
  &          & \dMapsto        &&&      & \rdDMapsto_{\alpha(\gamma'\gamma,x)} 
                                         & \dMapsto_{\alpha(\gamma',\gamma x)}\\
  &          & \gamma'\gamma x &&&      && f(\gamma'\gamma x) \\
\end{diagram}
\caption{The cocycle condition: $\alpha(\gamma'\gamma
  x)=\alpha(\gamma',\gamma x)\alpha(\gamma,x)$.\label{fig_comp}}
\end{figure}

When $f$ is actually action-preserving, that is, $f(\gamma
x)=\phi(\gamma)f(x)$ for some homomorphism
$\phi\from\Gamma\into\Lambda$, then we have
$\alpha(\gamma,x)=\phi(\gamma)$, so that $\alpha$ is independent of
the second coordinate.  Conversely, if $\alpha$ is independent of the
second coordinate then one can define
$\phi(\gamma)=\alpha(\gamma,\cdot)$ and the composition law implies
that $\phi$ is a homomorphism.  In this situation, the cocycle is said
to be \emph{trivial}.

In practice, when establishing rigidity one typically shows that an
arbitrary cocycle (arising from a homomorphism of orbits) is
equivalent to a trivial cocycle (which therefore arises from a
homomorphism of actions).  Here, we say that homomorphisms of orbits
$f,f'\from X\into Y$ are called \emph{equivalent} if there exists a
Borel function $b\from X\into\Lambda$ such that $f'(x)=b(x)f(x)$,
\emph{a.e}.  (That is, they lift the same function on the quotient
spaces $X/\Gamma\into Y/\Lambda$).  In this case, the corresponding
cocycles $\alpha,\alpha'$ are said to be cohomologous.  It is easy to
check that $f,f'$ are equivalent via $b$ iff the corresponding
cocycles $\alpha,\alpha'$ satisfy the relation
$\alpha'(\gamma,x)=b(\gamma x)\alpha(\gamma,x)b(x)^{-1}$ \emph{a.e.};
this is called the \emph{cohomology relation}.  The easiest way to see
that this is the case is to glance at Figure~\ref{fig_cohom}.

\begin{figure}[ht]
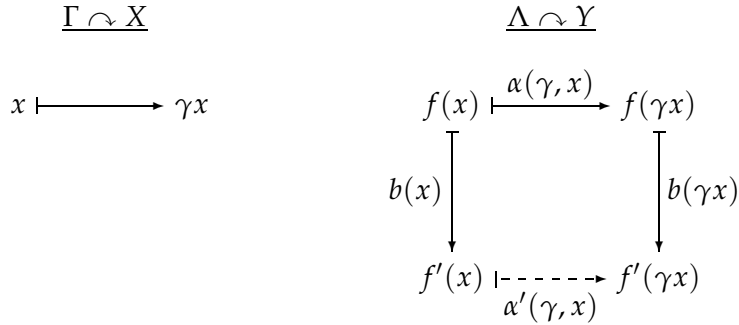

\begin{diagram}
  & \underline{\Gamma\actson X} & &&& & \underline{\Lambda\actson Y}\\
x & \rMapsto & \gamma x &&& f(x)            & \rMapsto^{\;\alpha(\gamma,x)\;}
                                                & f(\gamma x) \\
  &          &          &&& \dMapsto^{b(x)} &   & \dMapsto_{b(\gamma x)}\\
  &          &          &&& f'(x)           & \rDMapsto_{\alpha'(\gamma,x)}
                                                & f'(\gamma x) \\
\end{diagram}
\caption{The cohomology relation for cocycles:
  $\alpha'(\gamma,x)=b(\gamma
  x)\alpha(\gamma,x)b(x)^{-1}$.\label{fig_cohom}}
\end{figure}

We close this section by remarking that not all cocycles arise from
orbit-preserving maps.  An abstract \emph{cocycle} is any Borel
function satisfying the cocycle condition \emph{a.e.}, and two
cocycles are said to be \emph{cohomologous} if there exists a Borel
function $b$ satisfying the cohomology relation \emph{a.e.}  The most
powerful superrigidity results often have the conclusion that ``every
cocycle is cohomologous to a trivial cocycle.''  However, for most
applications there is no need for the extra strength gained by using
the abstract cocycle formulation.

\section{Ioana's theorem}

Cocycle superrigidity results were first established by Margulis and
Zimmer for cocycles $\Gamma\actson X\into\Lambda$ where $\Gamma$ is a
lattice in a higher rank Lie group acting ergodically on $X$.  These
results carried the additional hypothesis that $\Lambda$ is contained
in an algebraic group.  The first example of the most general form of
cocycle superrigidity, with the target $\Lambda$ arbitrary, was Popa's
result concerning Bernoulli actions.  In this section we shall discuss
Ioana's theorem, which establishes similar conclusions for profinite
actions.

Here, $\Gamma\actson X$ is said to be \emph{profinite} if as a
$\Gamma$-set, $X$ is the inverse limit of a family of finite
$\Gamma$-sets $X_n$.  In particular, there exist equivariant
projections $\pi_n\from X\into X_n$ and each element $x\in X$ can be
identified with the thread $(\pi_n(x))$.  We are interested in the
ergodic case; here, each $X_n$ is equipped with the uniform
probability measure and $\Gamma\actson X_n$ is transitive.

\begin{thm}[Ioana]
  \label{thm_ioana}
  Let $\Gamma\actson(X,\mu)$ be an ergodic, measure-preserving,
  profinite action, with invariant factor maps $\pi_n\from X\into
  X_n$.  Assume that $\Gamma$ has Property~(T).  Then for any cocycle
  $\alpha\from\Gamma\actson X\into\Lambda$, there exists $n$ and $a\in
  X_n$ such that the restriction of $\alpha$ to the action
  $\Gamma_a\actson\pi_n^{-1}(a)$ is cohomologous to a trivial cocycle.
\end{thm}

In other words, the conclusion is that $\Gamma\actson X$ is
``virtually superrigid'' in the sense that any orbit preserving map,
after it is restricted to a finite index component of the left-hand
side, comes from an action preserving map.  Ioana's theorem is
interesting when contrasted with Popa's theorem; while Bernoulli
actions are strongly mixing, profinite actions are highly non-mixing.
Indeed, for each $n$, $\Gamma$ just permutes the blocks
$\pi_n^{-1}(a)$, for $a\in X_n$, and it follows that $\bigcup_{a\in
  X_n}X_a\times X_a$ is a $\Gamma$-invariant subset of $X\times X$.

We remark that although our variant of Ioana's theorem is sufficient
for most applications, it is weaker than the state of the art in
several ways.  First, Ioana requires only that $\Gamma$ have the
\emph{relative} Property~(T) over some infinite normal subgroup $N$
such that $\Gamma/N$ is finitely generated.  Second, Ioana also shows
that $\alpha$ is equivalent to a cocycle defined on all of $X$.  Last,
Furman has generalized the statement by replacing profinite actions
with the more general class of compact actions.  ($\Gamma\actson X$ is
said to be \emph{compact} if when regarded as a subset of
$\aut(X,\mu)$ it is precompact in a suitable topology.)

\section{Cocycle untwisting}

We begin with the following preliminary result, which roughly speaking
says that if $\alpha\from\Gamma\actson X\into\Lambda$ is a cocycle,
and if for each $\gamma$ it has a ``very likely'' value, then $\alpha$
is cohomologous to the map which always takes on this likely value.
In particular, in this case $\alpha$ is cohomologous to a trivial
cocycle.

\begin{thm}
  \label{thm_untwist}
  Let $\Gamma\actson(X,\mu)$ be ergodic and measure-preserving, and
  let $\alpha\from\Gamma\actson X\into\Lambda$ be a cocycle.  Suppose
  that for all $\gamma\in\Gamma$ there exists
  $\lambda_\gamma\in\Lambda$ such that
  \[\mu\set{x\mid\alpha(\gamma,x)=\lambda_\gamma}\geq C>7/8\;.
  \]
  Then the map $\phi(\gamma)=\lambda_\gamma$ is a homomorphism and
  $\alpha$ is cohomologous to it.
\end{thm}

It is easy to see that $\phi$ must be a homomorphism: indeed, the
hypothesis guarantees that there is a non-null set of $x$ for which
$\phi(\gamma'\gamma)=\alpha(\gamma'\gamma,x)=\alpha(\gamma',\gamma
x)\alpha(\gamma,x)=\phi(\gamma')\phi(\gamma)$.  Hence it remains only
to establish the following result.

\begin{lem}
  \label{lem_close}
  Let $\Gamma\actson(X,\mu)$ be ergodic and measure-preserving, and
  let $\alpha,\beta\from\Gamma\actson X\into\Lambda$ be cocycles.
  Suppose that for all $\gamma\in\Gamma$
  \[\mu\set{x\mid\alpha(\gamma,x)=\beta(\gamma,x)}\geq C>7/8\;.
  \]
  Then $\alpha$ is cohomologous to $\beta$.
\end{lem}

We understand this result to say that if $\alpha$ and $\beta$ are
close in an $L^\infty$ sense, then they are cohomologous.  It follows
upon similar results of Popa and Furman, which draws the same
conclusion in the case that $\alpha$ and $\beta$ are close in an
appropriate $L^1$ sense (for instance, see
\cite[Theorem~4.2]{furman-popa}).  Ioana's proof, given below, may be
safely skipped until after reading the next section.

\begin{proof}[Proof of Lemma \ref{lem_close}]
  Let $\Gamma\actson X\times\Lambda$ be the action given by
  \[\gamma(x,\lambda)
    =\parens{\gamma x,\alpha(\gamma,x)\lambda\beta(\gamma,x)^{-1}}
  \]
  (this is an action thanks to the cocycle condition), and consider
  the corresponding left-regular representation.  The reason we use
  this representation is that $\alpha$ is close to $\beta$ iff a
  particular vector is close to being invariant.  Namely, let
  \[\xi=\chi_{X\times e}
  \]
  (read: the characteristic function of $X\times e$) and notice that
  \[\seq{\gamma\xi,\xi}=\mu\set{x\mid\alpha(\gamma,x)=\beta(\gamma,x)}\;.
  \]
  Using this together with the law of cosines, the hypothesis now
  translates to say that $\norm{\gamma\xi-\xi}\leq C<1/2$ for all
  $\gamma\in\Gamma$.  It is not difficult to see that this implies
  that there is an \emph{invariant} vector $\eta$ such that
  $\norm{\eta-\xi}<1/2$.  (Indeed, letting $S$ denote the convex hull
  of $\overline{\Gamma\cdot\xi}$, it is easily seen that there exists
  a unique vector $\eta\in S$ of minimal norm; this $\eta$ is
  necessarily invariant.)
  
  The idea for the conclusion of the proof is as follows.  If we had
  $\eta=\chi_{\mathsf{graph}(b)}$ for some function $b\from
  X\rightarrow\Lambda$, then we would be done.  Indeed, in this case
  the invariance of $\eta$ would mean that $b(x)=\lambda$ iff
  $b(\gamma x)=\alpha(\gamma,x)\lambda\beta(\gamma,x)^{-1}$, so that
  $b(\gamma x)=\alpha(\gamma,x)b(\gamma x)\beta(\gamma,x)^{-1}$.  In
  other words, $b$ would witness that $\alpha$ is cohomologous to
  $\beta$.  The fact that $\norm{\eta-\xi}<1/2$ implies that this is
  close to being the case.

  We actually define $b(x)=$ the $\lambda$ such that
  $\abs{\eta(x,\lambda)}>1/2$, if it exists and is unique.  The above
  computation shows that when $b(x),b(\gamma x)$ are both defined, the
  cohomology relation holds.  Moreover, the set where $b$ is defined
  is invariant, so by the ergodicity of $\Gamma\actson X$, it suffices
  to show that this set is non-null.  In fact, since $\eta$ and $\xi$
  are close, $b$ must take value $e$ on a non-null set:
  \begin{align*}
    \mu\set{x:\abs{\xi(x,e)-\eta(x,e)}\geq1/2}
    & \;\leq\; 4\int_{\set{x:\abs{\xi(x,e)-\eta(x,e)}\geq1/2}}\abs{\xi-\eta}^2\\
    & \;\leq\; 4\norm{\xi-\eta}^2\\
    & \;<\; 1
  \end{align*}
  This shows $\set{x:\abs{\eta(x,e)}>\frac12}$ is non-null, as
  desired.  A similar computation is used to show that with
  probability $1$, $e$ is the unique such element of $\lambda$.
\end{proof}

\section{Ioana's proof}

\subsection*{What we want}

We wish to find some $n$ and $a\in X_n$ such that for all
$\gamma\in\Gamma_a$,
\begin{equation}
  \label{eqn_want}
  (\mu_a\times\mu_a)\set{x,x'\mid\alpha(\gamma,x)=\alpha(\gamma,x')}\geq C>7/8
\end{equation}
where $\mu_a$ denotes the normalized restriction of $\mu$ to
$\pi_n^{-1}(a)$.  This would imply, by a straightforward computation,
that for each $\gamma\in\Gamma_a$ there exists a $\lambda\in\Lambda$
such that
\[\mu_a\set{x\mid\alpha(\gamma,x)=\lambda}\geq C>7/8
\]
and then we would be finished thanks to Theorem~\ref{thm_untwist}.

\subsection*{What we have}

Unfortunately, it is only immediately possible to obtain that the
quantities in Equation~\eqref{eqn_want} tend to $1$ on average, at a
rate depending on $\gamma$.  That is, for each $\gamma\in\Gamma$ we
have
\begin{equation}
  \label{eqn_have}
  \lim_{n\rightarrow\infty}\frac{1}{\abs{X_n}}\sum_{a\in X_n}(\mu_a\times\mu_a)
  \set{x,x'\mid\alpha(\gamma,x)=\alpha(\gamma,x')}=1\;.
\end{equation}
To see this, first note that it is equivalent to
\[\lim_{n\rightarrow\infty}\sum_{\lambda\in\Lambda}
  \parens{\frac{1}{\abs{X_n}}\sum_{a\in X_n}
   \mu_a\set{x\mid\alpha(\gamma,x)=\lambda}^2}=1\;.
\]
Now, we generally have that for any subset $S\subset X$,
\begin{equation}
  \label{eqn_limit}
  \lim_{n\rightarrow\infty}\frac{1}{\abs{X_n}}\sum_{a\in X_n}\mu_a(S)^2=\mu(S)\;.
\end{equation}
This is because the family $\{\chi_{\pi_n^{-1}(a)}\mid a\in
X_n,n\in\omega\}$ is dense in $L^2$, and while the right-hand side is
the norm-squared of $\chi_S$, the left-hand side is the norm-squared
of $\chi_S$ projected onto the span of $\{\chi_{\pi_m^{-1}(a)}\mid
a\in X_m,m\leq n\}$.  Finally, just apply Equation~\eqref{eqn_limit}
to each set $S=\set{x\mid\alpha(\gamma,x)=\lambda}$, and use the
dominated convergence theorem to pass the limit through the sum over
all $\lambda\in\Lambda$.

\subsection*{The proof}

The gap between what he have (the asymptotic information) and what we
want (the uniform information) is bridged by Property~(T).  Once again
the first step is to consider an appropriate representation, this time
one which compares the values of $\alpha(\gamma,x)$ as $x$ varies.
That is, we let $\Gamma\actson X\times X\times\Lambda$ by
\[\gamma(x,x',\lambda)=
  \parens{\gamma x,\gamma x',\alpha(\gamma,x)\lambda\alpha(\gamma,x')^{-1}}
\]
and consider the left-regular unitary representation corresponding to
this action.  The idea, very roughly, is that the degree to which
$\alpha(\gamma,x)$ is independent of $x$ will be measured by how close
a particular vector is to being $\Gamma$-invariant.

More precisely, for each $n$ define an orthonormal family of vectors
$\xi_a$ for $a\in X_n$ by
\begin{align*}
  \xi_a&=\abs{X_n}\cdot\chi_{\pi_n^{-1}(a)\times\pi_n^{-1}(a)\times e}\;,\\
  \intertext{and consider their normalized average}
  \xi_n&=\frac{1}{\sqrt{\abs{X_n}}}\sum_{a\in X_n}\xi_a\;,
\end{align*}
Then a simple calculation shows that
\begin{align*}
  \seq{\gamma\xi_a,\xi_a}
  &=(\mu_a\times\mu_a)\set{x,x'\mid\alpha(\gamma,x)=\alpha(\gamma,x')}
  \;,\text{ and}\\
  \seq{\gamma\xi_n,\xi_n}
  &=\frac{1}{\abs{X_n}}\sum_{a\in X_n}
  (\mu_a\times\mu_a)\set{x,x'\mid\alpha(\gamma,x)=\alpha(\gamma,x')}\;.
\end{align*}
So now, ``what we have'' and ``what we want'' can be translated as
follows: we have that the $\xi_n$ form a family of almost invariant
vectors, and we want a single $n$ and $a\in X_n$ such that $\xi_a$ is
nearly invariant, uniformly for all $\gamma\in\Gamma_a$.

The remainder of the argument is straightforward.  Since the $\xi_n$
form a family of almost invariant vectors, Property~(T) implies that
there exists $n$ and an invariant vector $\eta$ such that
$\norm{\eta-\xi_n}\leq\delta$.  Let $\eta'$ be the restriction of
$\eta$ to the set $\cup_{a\in
  X_n}\parens{\pi_n^{-1}(a)\times\pi_n^{-1}(a)\times\Lambda}$.  Since
this set is invariant, we have that $\eta'$ is invariant as well.
Since $\xi_n$ is supported on this set, we retain the property that
$\norm{\eta'-\xi_n}\leq\delta$.

Now, we simply express $\eta'$ as a normalized average of orthogonal
$\Gamma_a$-invariant vectors.  More specifically, write
\[\eta'=\frac{1}{\sqrt{\abs{X_n}}}\sum_{a\in X_a}\eta_a
\]
where $\eta_a$ is the appropriately rescaled restriction of $\eta'$ to
the set $\pi_n^{-1}(a)\times\pi_n^{-1}(a)\times\Lambda$.  Then by the
law of averages, we must have \emph{some} $a\in X_n$ such that
$\norm{\eta_a-\xi_a}\leq\delta$.  Moreover, $\eta_a$ is
$\Gamma_a$-invariant, so that for all $\gamma\in\Gamma_a$ we have
$\seq{\gamma\eta_a,\eta_a}=1$.  It follows that by an appropriate
choice of $\delta$, we can make $\seq{\gamma\xi_a,\xi_a}\geq C>7/8$
for all $\gamma\in\Gamma_a$.\qed

\section{Easy applications}

In this section, we use Ioana's theorem for one of its intended
purposes: to find many highly inequivalent actions.  The results
mentioned here are just meant to give the flavor of applications of
superrigidity; they by no means demonstrate the full power of the
theorem.  In the next section we will discuss the slightly more
interesting and difficult application to torsion-free abelian groups.
For further applications, see for instance \cite{ioana}, \cite{quasi},
\cite{sln}, and \cite{slocal}.

In searching for inequivalent actions, one might of course consider a
variety of inequivalence notions.  Here, we focus on just two of them:
orbit inequivalence and Borel incomparability.  Recall from the
introduction that $\Gamma\actson X$ and $\Lambda\actson Y$ are said to
be \emph{orbit equivalent} if there exists a measure-preserving and
orbit-preserving almost bijection from $X$ to $Y$.  Notice that this
notion depends only on the \emph{orbit equivalence relation} arising
from the two actions, and not on the actions themselves.  When this is
the case, we will often conflate the two, saying alternately that
certain actions are orbit equivalent or that certain equivalence
relations are ``orbit equivalent.''

Borel bireducibility is a purely set-theoretic notion with its origins
in logic.  The connection is that if $E$ is an equivalence relation on
a standard Borel space $X$, then we can think of $E$ as representing a
classification problem.  For instance, if $X$ happens to be a set of
codes for a family of structures, then studying the classification of
those structures amounts to studying the isomorphism equivalence
relation $E$ on $X$.  We refer the reader to \cite{gao} for a complete
introduction to the subject.

If $E$ and $F$ are equivalence relations on $X$ and $Y$, then $E$ is
said to be \emph{Borel reducible} to $F$ if and only if there exists a
Borel function $f\from X\into Y$ satisfying $x\mathrel{E}x'$ iff
$f(x)\mathrel{F}f(x')$.  We think of this as saying that the
classification problem for elements of $X$ up to $E$ is no more
complex than the classification problem for elements of $Y$ up to $F$.
Thus, if $E$ and $F$ are Borel bireducible (that is, there is a
reduction both ways), then they represent classification problems of
the same complexity.

It is elementary to see that neither of orbit equivalence and Borel
bireducibility implies the other.  For instance, given any
$\Gamma$-space $X$ one can form a disjoint union $X\sqcup X'$, where
$X'$ is a $\Gamma$-space of very high complexity which is declared to
be of measure $0$.  Conversely, if $X$ is an ergodic and hyperfinite
$\Gamma$-space, then it is known that it is bireducible with $X\sqcup
X$, but the two cannot be orbit equivalent.  It is even possible,
without much more difficulty, to find \emph{two} ergodic actions which
are bireducible but not orbit equivalent.

We are now ready to begin with the following direct consequence of
Ioana's theorem.  It was first established by Simon Thomas in
connection working on classification problem for torsion-free abelian
groups of finite rank.  His proof used Zimmer's superrigidity theorem
and some additional cocycle manipulation techniques; with Ioana's
theorem in hand, the proof will be much simpler.

\begin{cor}
  \label{cor_free}
  If $n\geq3$ is fixed and $p,q$ are primes such that $p\neq q$, then
  the actions of $\SL_n(\ZZ)$ on $\SL_n(\ZZ_p)$ and $\SL_n(\ZZ_q)$ are
  orbit inequivalent and Borel incomparable.
\end{cor}

Here, $\ZZ_p$ denotes the ring of $p$-adic integers.  It is easy to
see that $\SL_n(\ZZ)\actson\SL_n(\ZZ_p)$ is a profinite action, being
the inverse limit of the actions $\SL_n(\ZZ)\actson\SL_n(\ZZ/p^i\ZZ)$
together with their natural system of projections.

\begin{proof}
  Let $p\neq q$ and suppose that $f$ is either a either an orbit
  equivalence or a Borel reduction from
  $\SL_n(\ZZ)\actson\SL_n(\ZZ_p)$ to $\SL_n(\ZZ)\actson\SL_n(\ZZ_q)$.
  We now apply Ioana's theorem together with the understanding of
  cocycles gained in the previous section.  The conclusion is: we can
  suppose without loss of generality that there exists a finite index
  subgroup $\Gamma_0\leq\SL_n(\ZZ)$, a $\bar\Gamma_0$-coset
  $X\subset\SL_n(\ZZ_p)$, and a homomorphism
  $\phi\from\Gamma_0\into\SL_n(\ZZ)$ which makes $f$ into an
  action-preserving map from $\Gamma_0\actson X$ into
  $\SL_n(\ZZ)\actson\SL_n(\ZZ_q)$.

  Now, in the measure-preserving case, it is not difficult to conclude
  that $f$ is a ``virtual isomorphism'' between the two actions.  We
  claim that this can be achieved even in the case that $f$ is just a
  Borel reduction.  First, we can assume that $\phi$ is an embedding.
  Indeed, by Margulis's theorem on normal subgroups
  \cite[Theorem~8.1.2]{zimmer}, either $\im(\phi)$ or $\ker(\phi)$ is
  finite.  If $\ker(\phi)$ is finite, then we can replace $\Gamma_0$
  by a finite index subgroup (and $X$ by a coset of the new
  $\bar\Gamma_0$) to suppose that $\phi$ is injective.  On the other
  hand, if $\im(\phi)$ is finite, then we can replace $\Gamma_0$ by a
  finite index subgroup to suppose that $\phi$ is trivial.  But this
  would mean that $f$ is $\Gamma_0$-invariant, and so by ergodicity of
  $\Gamma_0\actson X$, $f$ would send a conull set to a single point,
  contradicting that $f$ is countable-to-one.

  Second, $\phi(\Gamma_0)$ must be a finite index subgroup of
  $\SL_n(\ZZ)$.  Indeed, by Margulis's superrigidity theorem, $\phi$
  can be lifted to an isomorphism of $\SL_n(\RR)$, and it follows that
  $\phi(\Gamma_0)$ is a lattice of $\SL_n(\RR)$.  But then it is easy
  to see that any lattice which is contained in $\SL_n(\ZZ)$ must be
  commensurable with $\SL_n(\ZZ)$.

  Third, by the ergodicity of $\Gamma_0\actson X$, we can assume that
  $\im(f)$ is contained in a single $\overline{\phi(\Gamma_0)}$ coset
  $Y_0$.  And now because $\phi(\Gamma_0)$ preserves a unique measure
  on $Y_0$ (the Haar measure), and because $\phi(\Gamma_0)$ preserves
  $f_*(\mathrm{Haar})$, we actually conclude that $f$ is
  measure-preserving.  In summary, we have shown that $(\phi,f)$ is a
  measure and action-preserving isomorphism between $\Gamma_0\actson
  X_0$ and $\phi(\Gamma_0)\actson Y_0$, which establishes the claim.

  Finally, a short computation confirms the intuitive, algebraic fact
  that the existence of such a map is ruled out by the mismatch in
  primes between the left and right-hand side.  We give just a quick
  sketch; for a few more details see \cite[Section~6]{super}.  Now, it
  is well-known that there are constants $A_p$ such that for any
  $\Delta\leq\SL_n(\ZZ)$ of finite index, the index of $\bar\Delta$ in
  $\SL_n(\ZZ_p)$ divides $A_pp^r$ for some $r$.  It follows that if
  $\Delta\leq\Gamma_0$, then $X$ breaks up into some number $N$ of
  ergodic $\Delta$-sets with $N\mid A_pp^r$.  Since $(\phi,f)$ is a
  measure and action-preserving isomorphism, we also have that $Y$
  breaks up into $N$ ergodic $\phi(\Delta)$ sets, and hence $N\mid
  A_qq^s$ also.  But it is not difficult to choose $\Delta$ small
  enough to ensure that $N$ is large enough for this to be a
  contradiction.
\end{proof}

This argument can be easily generalized to give uncountably many
incomparable actions of $\SL_n(\ZZ)$.  Given an infinite set $S$ of
primes with increasing enumeration $S=\set{p_i}$, we can construct a
profinite $\SL_n(\ZZ)$-set
\[K_S=\lim_{\leftarrow}\SL_n(\ZZ/p_1\cdots p_i\ZZ)\;.
\]
It is not much more difficult to show (as Ioana does) that when
$\abs{S\mathrel{\triangle}S'}=\infty$, the actions $\SL_n(\ZZ)\actson
K_S$ and $\SL_n(\ZZ)\actson K_{S'}$ are orbit inequivalent.  In fact,
this shows that there are ``$E_0$ many'' orbit inequivalent profinite
actions of $\SL_n(\ZZ)$.  Of course, it is known from different
arguments (exposited in \cite[Theorem~17.1]{kechris-global}) that the
relation of orbit equivalence on the ergodic actions of $\SL_n(\ZZ)$
is very complex (for instance not Borel).  But the methods used here
give us more detailed information: we have an explicit family of
inequivalent actions, the actions are special (they are classical and
profinite), and what's more they are Borel incomparable.

So far, we have considered only free actions of $\SL_n(\ZZ)$.  But if
one just wants to use Ioana's theorem to find orbit inequivalent
actions, it is enough to consider actions which are just free almost
everywhere.  Here, a measure-preserving action $\Gamma\actson X$ is
said to be \emph{free almost everywhere} if the set
$\set{x\mid\gamma\neq1\rightarrow\gamma x\neq x}$ is conull (that is,
the set where $\Gamma$ acts freely is conull).

Unfortunately, in the purely Borel context it is not sufficient to
work with actions which are free almost everywhere, since in this case
we are not allowed to just delete a null set on the right-hand side.
The next result shows how to get around this difficulty.  Once again,
it was originally obtained by Simon Thomas using Zimmer's
superrigidity theorem.

\begin{cor}
  \label{cor_nonfree}
  If $n\geq3$ is fixed and $p,q$ are primes with $p\neq q$, then the
  actions of $\SL_n(\ZZ)$ on $\PP(\QQ_p^n)$ and $\PP(\QQ_q^n)$ are are
  orbit inequivalent and Borel incomparable.
\end{cor}

Here, $\PP(\QQ_p^n)$ denotes projective space of lines through
$\QQ_p^n$.  Since $\PP(\QQ_p^n)$ is a transitive $\SL_n(\ZZ_p)$-space,
this result is quite similar to the last one.  We note also that while
$\SL_n(\ZZ)$ does not act freely on $\PP(\QQ_p^n)$, it does act freely
on a conull subset \cite[Lemma~6.2]{super}.

\begin{proof}
  First suppose that $f\from\PP(\QQ_p^n)\into\PP(\QQ_q^n)$ is a
  measure-preserving and orbit-preserving map.  Then we can simply
  restrict the domain of $f$ to assume that it takes values in the
  part of $\PP(\QQ_q^n)$ where $\SL_n(\ZZ)$ acts freely.  Afterwards,
  we can obtain a contradiction using essentially the same
  combinatorial argument as in the proof of Corollary~\ref{cor_free}.

  The proof in the case of Borel reducibility requires an extra step.
  Namely, we cannot be sure that $f$ sends a conull set into the part
  of $\PP(\QQ_q^n)$ where $\SL_n(\ZZ)$ acts freely.  However, if it
  does not, then by the ergodicity of $\SL_n(\ZZ)\actson\PP(\QQ_p^n)$
  we can assume that $f$ sends a conull set into the part of
  $\PP(\QQ_q^n)$ where $\SL_n(\ZZ)$ acts \emph{non}-freely.  Our aim
  will be to show that this assumption leads to a contradiction.

  Proceeding, let us assume that there exists a conull subset
  $X\subset\PP(\QQ_p^n)$ such that for all $x\in X$, there exists
  $\gamma\neq1$ such that $\gamma f(x)=f(x)$.  Then for all $x\in X$,
  $f(x)$ lies inside a nontrivial eigenspace of some element of
  $\SL_n(\ZZ)$.  Hence, if we let $V_x$ denote the minimal subspace of
  $\QQ_q^n$ which is defined over $\bar\QQ$ such that $f(x)\subset
  V_x$, then $V_x$ is necessarily nontrivial.

  Note that since $\bar\QQ$ is countable, there are only countably
  many possibilities for $V_x$.  Hence, there exists a non-null subset
  $X'$ of $X$ and a fixed subspace $V$ of $\QQ_q^n$ such that for all
  $x\in X'$, we have $V_x=V$.  By the ergodicity of $\SL_n(\ZZ)\actson
  X$, the set $X''=\SL_n(\ZZ)\cdot X'$ is conull, and it follows that
  we can adjust $f$ to assume that for all $x\in X''$ we have $V_x=V$.
  (More precisely, replace $f(x)$ by $f'(x)=f(\gamma x)$, where
  $\gamma$ is the first element of $\SL_n(\ZZ)$ such that $\gamma x\in
  X''$.)

  Now, let $H\leq\GL(V)$ denote the group of projective linear
  transformations induced on $V$ by $\SL_n(\ZZ)_{\set{V}}$.  It is an
  easy exercise, using the minimality of $V$, to check that $H$ acts
  freely on $\PP(V)$, and that $f$ is a homomorphism of orbits from
  $\SL_n(\ZZ)\actson X''$ into $H\actson\PP(V)$.  Admitting this, we
  can finally apply Ioana's theorem to suppose that there exists a
  finite index subgroup $\Gamma_0\leq\Gamma$ and a nontrivial
  homomorphism $\phi\from\Gamma_0\into H$.  As in the proof of
  Corollary~\ref{cor_free}, we can suppose that $\phi$ is an
  embedding.  We thus get a contradiction from the next result, below.
\end{proof}

\begin{thm}
  If $\Gamma_0\leq\SL_n(\ZZ)$ is a subgroup of finite index and
  $\bm{G}$ is an algebraic $\bar\QQ$-group with $\dim(\bm{G})<n^2-1$,
  then $\Gamma_0$ does not embed into $\bm{G}(\bar\QQ)$.\qed
\end{thm}

The idea of the proof is to apply Margulis's superrigidity theorem.
That is, one wishes to conclude that such an embedding lifts to some
kind of rational map $\SL_n(\RR)\into\bm{G}$, a clear dimension
contradiction.  However, a little extra work is needed to handle the
case of a $\bar\QQ$-group on the right-hand side (see
\cite[Theorem~4.4]{slocal}).

\section{Torsion-free abelian groups of finite rank}

The torsion-free abelian groups of rank $1$ were classified by Baer in
1937.  The next year, Kurosh and Malcev expanded on his methods to
give classifications for the torsion-free abelian groups of ranks $2$
and higher.  Their solution, however, was considered inadequate
because the invariants they provided were no easier to distinguish
than the groups themselves.

In 1998, Hjorth proved, using methods from the study of Borel
equivalence relations, that the classification problem for rank $2$
torsion-free abelian groups is \emph{strictly harder} than that for
rank $1$ (see \cite{hjorth-classification}).  However, his work did
not answer the question of whether the classification problem for rank
$2$ groups is as complex as for all finite ranks, or whether there is
more complexity to be found by looking at ranks $3$ and higher.

Let $R(n)$ denote the space of torsion-free abelian groups of rank
exactly $n$, that is, the set of full-rank subgroups of $\QQ^n$.  Let
$\oiso_n$ denote the isomorphism relation on $R(n)$.  In this section
we shall give a concise and essentially self-contained proof of
Thomas's theorem:

\begin{thm}[Thomas, \cite{torsionfree}]
  \label{thm_thomas}
  For $n\geq2$, we have that $\oiso_n$ lies properly below
  $\oiso_{n+1}$ in the Borel reducibility order.
\end{thm}

Thomas's original argument used Zimmer's superrigidity theorem.  In
this presentation, we have essentially copied his argument verbatim,
with a few simplifications stemming from the use of Ioana's theorem
instead of Zimmer's theorem.

The first connection between is this result and the results of the
last section is that for $A,B\in R(n)$, we have $A\iso B$ iff there
exists $g\in\GL_n(\QQ)$ such that $B=g(A)$.  Hence, the isomorphism
relation $\oiso_n$ is given by a natural action of the linear group
$\GL_n(\QQ)$.  Unfortunately, even restricting to just the action of
$\SL_n(\ZZ)$, the space $R(n)$ is nothing like a profinite space.

\subsection*{The Kurosh--Malcev invariants}

Although I have said that the Kurosh--Malcev invariants do not
adequately classify the torsion-free abelian groups of finite rank, we
will get around our difficulties by working with the Kurosh--Malcev
invariants rather than with the original space $R(n)$.  The following
is the key result concerning the invariants; see
\cite[Chapter~93]{fuchs} for a full account.

\begin{thm}[Kurosh, Malcev]
  \label{thm_km}
  The map $A\mapsto A_p=\ZZ_p\otimes A$ is a $\GL_n(\QQ)$-preserving
  bijection between the (full rank) $p$-local subgroups of $\QQ^n$
  and the (full rank) $\ZZ_p$-submodules of $\QQ_p^n$.  The inverse
  map is given by $A_p\mapsto A=A_p\cap\QQ^n$.
\end{thm}

Here, a subgroup of $\QQ^n$ is said to be \emph{$p$-local} if it is
infinitely $q$-divisible for each prime $q\neq p$.  Kurosh and Malcev
proved that a subgroup $A\leq\QQ^n$ is determined by the sequence
$(A_p)$; this sequence is said to be the Kurosh--Malcev invariant
corresponding to $A$.  It follows of course that $A$ is determined up
to isomorphism by the orbit of $(A_p)$ under the coordinatewise action
of $\GL_n(\QQ)$.  (It is now easy to see why these invariants serve as
a poor classification: such orbits can be quite complex.)  All that we
shall need from this classification is the following corollary.

\begin{prop}
  \label{prop_projred}
  There exists a Borel reduction from $\GL_n(\QQ)\actson\PP(\QQ_p^n)$
  to $\oiso_n$.
\end{prop}

Since $\GL_n(\QQ)\actson\PP(\QQ_p^n)$ is closely related to a profinite
action, Proposition~\ref{prop_projred} will eventually enable us to
apply Ioana's theorem in the proof of Theorem~\ref{thm_thomas}.

\begin{proof}[Sketch of proof]
  Given a linear subspace $V\leq\QQ_p^n$, let $V^\perp$ denote its
  orthogonal complement.  Then there exists a vector $v$ such that
  $V^\perp\oplus\ZZ_pv$ is a full-rank submodule of $\QQ_p^n$.  By
  Theorem~\ref{thm_km}, this module corresponds to an element $f(V)\in
  R(n)$.  This is how the Kurosh--Malcev construction is used.

  To verify that it works, one uses the fact that the Kurosh--Malcev
  construction is $\GL_n(\QQ)$-preserving, together with the technical
  fact: if $\dim W=\dim W'=n-1$ and $W\oplus\ZZ_pw,W'\oplus\ZZ_pw'$
  are full-rank modules, then $W'=gW$ for some $g\in\GL_n(\QQ)$
  actually implies that $W'\oplus\ZZ_pw'=g(W\oplus\ZZ_pw)$ for some
  $g\in\GL_n(\QQ)$.
\end{proof}

\subsection*{The problem of freeness}

Suppose now that $n\geq2$ and that there exists a Borel reduction from
$\oiso_{n+1}$ to $\oiso_n$.  By Proposition~\ref{prop_projred}, there
exists a profinite, ergodic $\SL_{n+1}(\ZZ)$-space $X$ (namely
$X=\PP(\QQ_p^{n+1})$) and a countable-to-one homomorphism of orbits
$f$ from $\SL_{n+1}(\ZZ)\actson X$ to $\oiso_n$.  We can almost apply
Ioana's theorem, except that unfortunately $\oiso_n$ is not induced by
a free action of any group.  The following simple observation gives us
an approach for getting around this difficulty.

\begin{prop}
  \label{prop_free}
  Let $f$ be a homomorphism of orbits from $\Gamma\actson X$ into
  $\Lambda\actson Y$.  Suppose that there exists a fixed
  $K\leq\Lambda$ such that for all $x\in X$, $\stab_\Lambda(f(x))=K$.
  Then $N_\Lambda(K)/K$ acts freely on $f(X)$, and $f$ is a homomorphism
  of orbits from $\Gamma\actson X$ into $N_\Lambda(K)/K\actson f(X)$.
\end{prop}

\begin{proof}
  By definition, we have that $N_\Lambda(K)/K$ acts on $f(X)$ by
  $\lambda K\cdot y=\lambda y$.  The action is free because $\lambda
  y=y$ implies that $\lambda\in K$.  To see that $f$ is still a
  homomorphism of orbits, just note that if $f(x')=\lambda f(x)$ then
  since $\stab_\Lambda(f(x))=\stab_\Lambda(f(x'))=K$, then we must
  have that $\lambda$ normalizes $K$.
\end{proof}

One can now formulate a strategy for proving Thomas's theorem along
the following lines:

\begin{claimI}
  By passing to a conull subset of $X$, we can assume without loss of
  generality that for all $x$ we have $\stab_{\GL_n(\QQ)}(f(x))=$ some
  fixed $K$.
\end{claimI}

\begin{claimII}
  There cannot exist a nontrivial homomorphism from (a finite index
  subgroup of) $\SL_{n+1}(\ZZ)$ into $N_{\GL_n(\QQ)}(K)/K$.
\end{claimII}

This would yield a contradiction, since by Proposition
\ref{prop_projred} and Claim~I, Ioana's theorem would provide the
nontrivial homomorphism ruled out in Claim~II.  Unfortunately, this
approach doesn't turn out to be a good one.  The reason is that
Claim~I seems to be as difficult to prove as Theorem~\ref{thm_thomas}
itself.  Moreover, Claim~II is not known to be true in this
generality.  (In fact, Claim~I has recently been established by Thomas
in \cite{slocal}, but his proof actually requires all of the arguments
below and more.)

\subsection*{Use quasi-isomorphism instead}

To reduce the number of possibilities for $\stab(f(x))=\aut(f(x))$, we
change categories from isomorphism to quasi-isomorphism.  We say that
groups $A,B\leq\QQ^n$ are \emph{quasi-isomorphic}, written
$A\qiso_nB$, if $B$ is commensurable with an isomorphic copy of $A$.
Of course, $\oqiso_n$ is a courser relation than $\oiso_n$, but it is
easy to check that it is still a countable Borel equivalence relation
(indeed, the commensurability relation is a countable relation in this
case, see \cite[Lemma~3.2]{torsionfree}).  Hence, the map $f$ from
above is again a countable-to-one Borel homomorphism from
$\SL_{n+1}(\ZZ)\actson X$ to $\oqiso_n$.

Now, rather than attempting to fix the automorphism group of $f(x)$,
we shall fix the quasi-endomorphism ring $\QEnd(A)$ of $f(x)$.  Here,
if $A\leq\QQ^n$ then $g\in\GL_n(\QQ)$ is said to be a
\emph{quasi-endomorphism} of $A$ if $\phi(A)$ is commensurable with a
subgroup of $A$.  (Equivalently, $n\phi(A)\subset A$ for some
$n\in\NN$.)  Then unlike $\End(A)$, it is clear that $\QEnd(A)$ is a
$\QQ$-subalgebra of $\mathrm{M}_{n\times n}(\QQ)$.  It follows that
there are just countably many possibilities for $\QEnd(f(x))$, since
an algebra is determined by any $\QQ$-vector space basis for it.
Hence, there exists $K$ such that $\QEnd(f(x))=K$ for a nonnull set of
$x$.  Arguing as in the proof of Corollary~\ref{cor_nonfree}, we may
replace $X$ by a conull subset and adjust $f$ to assume that for all
$x\in X$, we have $\QEnd(f(x))=K$.

Thus, we have successfully obtained our analog of Claim~I for
quasi-isomorphism.  Indeed, copying the arguments in the proof of
Proposition~\ref{prop_free}, we see that $f$ is a homomorphism
\[f\from\SL_{n+1}(\ZZ)\actson X\longrightarrow
  N_{\GL_n(\QQ)}(K)/K^\times\actson f(X)
\]
and that $N_{\GL_n(\QQ)}(K)/K^\times$ acts freely on $f(X)$.  We may
therefore apply Ioana's theorem to suppose that there exists a finite
index subgroup $\Gamma_0\leq\PSL_{n+1}(\ZZ)$, a positive measure
$X_0\subset X$, and a homomorphism $\phi\from\Gamma_0\into
N_{\GL_n(\QQ)}(K)/K^\times$ such that for $x\in X_0$ and
$\gamma\in\Gamma$, we have
\[f(\gamma x)=\phi(\gamma)f(x)\;.
\]
Note that $\phi$ must be nontrivial, since if $\phi(\Gamma_0)=1$, then
this says that $f$ is $\Gamma_0$-invariant.  But then, by ergodicity
of $\Gamma_0\actson X_0$, $f$ would send a conull set to one point,
contradicting that $f$ is countable-to-one.

\subsection*{A dimension contradiction}

The set theory is now over; we have only to establish the algebraic
fact that the analog of Claim~II holds: there does not exist a
nontrivial homomorphism from $\Gamma_0$ into
$N_{\GL_n(\QQ)}(K)/K^\times$.  Again by Margulis's theorem on normal
subgroups, we can suppose that $\phi$ is an embedding.  Then using
Margulis's superrigidity theorem, it suffices to show that
$N_{\GL_n(\QQ)}(K)/K^\times$ is contained in an algebraic group of
dimension strictly smaller than $\dim(\bm{\PSL{}}_{n+1})=(n+1)^2-1$.

To see this, first note that since the subalgebra $K$ of
$\mathrm{M}_{n\times n}(\QQ)$ is definable from a vector space basis,
we have that $K=\bm{K}(\QQ)$, where $\bm{K}$ is an algebraic
$\QQ$-group inside $\mathrm{M}_{n\times n}$.  Basic facts from
algebraic group theory imply that $N_{\GL_n(\QQ)}(K)=\bf N(\QQ)$ and
$K^\times=\bm{K'}(\QQ)$, where again $\bm{N},\bm{K'}$ are algebraic
$\QQ$-groups inside $\mathrm{M}_{n\times n}$.  Finally,
$N_{\GL_n(\QQ)}(K)/K^\times$ is exactly $\bm{N}(\QQ)/\bm{K'}(\QQ)$,
which is contained in the algebraic $\QQ$-group $\bm{N}/\bm{K'}$.
Since the dimension of an algebraic group decreases when passing to
subgroups and quotients, we have
\[\dim(\bm{N}/\bm{K'})\leq\dim(\mathrm{M}_{n\times
  n})=n^2<(n+1)^2-1\;,
\]
as desired.  This completes the proof.\qed

\bibliographystyle{alpha}
\begin{singlespace}
  \bibliography{ioana2}

\begin{thebibliography}{Tho03b}

\bibitem[Cos10]{sln}
Samuel Coskey.
\newblock Borel reductions of profinite actions of {${\rm SL}\sb
  n(\mathbb{Z})$}.
\newblock {\em Ann. Pure Appl. Logic}, 161(10):1270--1279, 2010.

\bibitem[Cos12]{quasi}
Samuel Coskey.
\newblock The classification of torsion-free abelian groups of finite rank up
  to isomorphism and up to quasi-isomorphism.
\newblock {\em Trans. Amer. Math. Soc.}, 364(1):175--194, 2012.

\bibitem[Fuc73]{fuchs}
L{\'a}szl{\'o} Fuchs.
\newblock {\em Infinite abelian groups. {V}ol. {II}}.
\newblock Academic Press, New York, 1973.
\newblock Pure and Applied Mathematics. Vol. 36-II.

\bibitem[Fur99]{furman-oe}
Alex Furman.
\newblock Orbit equivalence rigidity.
\newblock {\em Ann. of Math. (2)}, 150(3):1083--1108, 1999.

\bibitem[Fur07]{furman-popa}
Alex Furman.
\newblock On {P}opa's cocycle superrigidity theorem.
\newblock {\em Int. Math. Res. Not. IMRN}, (19):Art. ID rnm073, 46, 2007.

\bibitem[Gao09]{gao}
Su~Gao.
\newblock {\em Invariant descriptive set theory}, volume 293 of {\em Pure and
  Applied Mathematics (Boca Raton)}.
\newblock CRC Press, Boca Raton, FL, 2009.

\bibitem[Hjo99]{hjorth-classification}
Greg Hjorth.
\newblock Around nonclassifiability for countable torsion free abelian groups.
\newblock In {\em Abelian groups and modules ({D}ublin, 1998)}, Trends Math.,
  pages 269--292. Birkh\"auser, Basel, 1999.

\bibitem[Ioa07]{ioanathesis}
Adrian Ioana.
\newblock {\em Some rigidity results in the orbit equivalence theory of
  non-amenable groups}.
\newblock ProQuest LLC, Ann Arbor, MI, 2007.
\newblock Thesis (Ph.D.)--University of California, Los Angeles.

\bibitem[Ioa11]{ioana}
Adrian Ioana.
\newblock Cocycle superrigidity for profinite actions of property ({T}) groups.
\newblock {\em Duke Math. J.}, 157(2):337--367, 2011.

\bibitem[Kec10]{kechris-global}
Alexander~S. Kechris.
\newblock {\em Global aspects of ergodic group actions}, volume 160 of {\em
  Mathematical Surveys and Monographs}.
\newblock American Mathematical Society, Providence, RI, 2010.

\bibitem[Kid08]{kida-oe}
Yoshikata Kida.
\newblock Orbit equivalence rigidity for ergodic actions of the mapping class
  group.
\newblock {\em Geom. Dedicata}, 131:99--109, 2008.

\bibitem[MS06]{monodshalom}
Nicolas Monod and Yehuda Shalom.
\newblock Orbit equivalence rigidity and bounded cohomology.
\newblock {\em Ann. of Math. (2)}, 164(3):825--878, 2006.

\bibitem[Pop07]{popa_paper}
Sorin Popa.
\newblock Cocycle and orbit equivalence superrigidity for malleable actions of
  {$w$}-rigid groups.
\newblock {\em Invent. Math.}, 170(2):243--295, 2007.

\bibitem[Tho03a]{torsionfree}
Simon Thomas.
\newblock The classification problem for torsion-free abelian groups of finite
  rank.
\newblock {\em J. Amer. Math. Soc.}, 16(1):233--258 (electronic), 2003.

\bibitem[Tho03b]{super}
Simon Thomas.
\newblock Superrigidity and countable {B}orel equivalence relations.
\newblock {\em Ann. Pure Appl. Logic}, 120(1-3):237--262, 2003.

\bibitem[Tho11]{slocal}
Simon Thomas.
\newblock The classification problem for {$S$}-local torsion-free abelian
  groups of finite rank.
\newblock {\em Adv. Math.}, 226(4):3699--3723, 2011.

\bibitem[Zim84]{zimmer}
Robert~J. Zimmer.
\newblock {\em Ergodic theory and semisimple groups}, volume~81 of {\em
  Monographs in Mathematics}.
\newblock Birkh\"auser Verlag, Basel, 1984.

\end{thebibliography}
\end{singlespace}

\end{document}